\documentclass[10pt]{article}
\usepackage{latexsym, amsmath, amsfonts, amscd, amssymb, verbatim,mathrsfs}

\usepackage{color}
\usepackage[colorinlistoftodos]{todonotes}

\usepackage[margin=1in]{geometry}  

\newtheorem{Theorem}{Theorem}[section]

\newtheorem{Remark}{Remark}[section]

\newtheorem{Lemma}{Lemma}[section]

\newtheorem{Definition}{Definition}[section]

\normalsize

\makeatletter
\renewcommand\@biblabel[1]{#1.}
\makeatother

\begin{document}

\title{\bf Characterizations of Nonsmooth Robustly Quasiconvex Functions}

\author{Hoa T. Bui\footnote{Centre for Informatics and Applied Optimization, Faculty of Science and Technology, Federation University Australia, POB	663, Ballarat, Vic, 3350, Australia.
		E-mail: h.bui@federation.edu.au}, Pham  Duy Khanh\footnote{Department of Mathematics, HCMC University of Pedagogy, 280 An Duong Vuong, Ho Chi Minh, Vietnam and Center for Mathematical
		Modeling, Universidad de Chile, Beauchef 851, Edificio Norte, Piso 7, Santiago, Chile. E-mails:
		pdkhanh182@gmail.com; pdkhanh@dim.uchile.cl},Tran Thi Tu Trinh\footnote{Department of Mathematics and Statistics, Oakland University, 318 Meadow Brook Rd, Rochester, MI 48309, USA. Email:
		thitutrinhtran@oakland.edu}} \maketitle
\date{}

\medskip
\begin{quote}
\noindent {\bf Abstract} Two criteria for the robust quasiconvexity  of lower semicontinuous functions  are established in terms of Fr\'echet  subdifferentials in Asplund spaces.


\end{quote}
\section{Introduction}
The question of characterizing convexity and generalized convexity properties in terms of subdifferentials  receives tremendous attention in optimization theory and variational analysis.  For decades, there have been received many significant contributions devoted to this question such as \cite{Clarke83,CorreaJofreThibault,Mordukhovich06,Poliquin90,Rock70} for convex functions, \cite{AusCorMar94,AusCorMar95,Aus98,BarronGoebelJensen121,Luc93,PenotQuang97} for quasiconvex functions and \cite{BarronGoebelJensen121} for robustly quasiconvex functions. 

\medskip
This paper follows this stream of research. Our aim is to establish the first-order characterizations for the robust quasiconvexity of lower semicontinuous functions in Asplund spaces. First,  some existing results regarding to the properties of subdifferential operators of convex, quasiconvex functions are recalled in Section 2, where the definitions and some basic results are given as well. Besides, necessary and sufficient first-order conditions for a lower semicontinuous function to be quasiconvex are reconsidered. Those characterizations moreover could be used to characterize the Asplund property of the given space. Second, two criteria for the robust quasiconvexity  of lower semicontinuous functions in Asplund spaces  are obtained by using  Fr\'echet  subdifferentials  in Section 3. Each criterion corresponds to each type of analogous conditions for quasiconvexity. The first one is based on the zero and first order condition for quasiconvexity (see Theorem~\ref{fod}(b) in Section 2).  It extends \cite[Proposition~5.3]{BarronGoebelJensen121} from finite dimensional spaces to Asplund spaces. Moreover, its proof also overcomes a glitch in the proof of the sufficient condition of \cite[Proposition~5.3]{BarronGoebelJensen121}. The second criterion is totally new. It is settled from the equivalence of the quasiconvexity of lower semicontinuous functions and the quasimonotonicity
 of their subdifferential operators (see Theorem~\ref{fod}(c) in Section 2).

\section{Preliminaries}
Let $X$ be a Banach space and $X^*$ its dual space. 
$X$ is called an Asplund space, or has the Asplund property, if  every separable subspace $Y$ of $X$ has separable continuous dual space $Y^*$.
The duality pairing on $X\times X^*$ is denoted by $\langle .,.\rangle$. 
In what follows, $\overline{\mathbb{R}}:=]-\infty,\infty]$;  $\mathbb{B}_r(x)$ is the open ball of radius $r>0$ centered at $x\in X$ and 
$\mathbb{B^*}\subset X^*$ is the closed ball of radius $1$ centered at $0_{X^*}$. 
The extended real-valued function $\varphi: X\to \overline{\mathbb{R}}$ considered mostly is proper lower semicontinuous (l.s.c), i.e. $\varphi$ is not identically $+\infty$, and the lower level sets $\varphi^{\le}_\alpha:=\{x\in X: \varphi(x)\le \alpha\}$ are closed for all $\alpha\in \mathbb{R}$. As usual dom$\varphi$ stands for the domain of $\varphi$, defined as 
$$
\text{dom}\varphi:=\{x\in X: \varphi(x)<\infty\}.
$$
For a set-valued mapping $A:X\rightrightarrows X^*$, the domain of $A$ is written
$$
\text{dom}A:=\{x\in X: A(x)\neq \emptyset\}.
$$
The graphs of $\varphi$ and $A$ are respectively defined as
 $$
 \text{graph}\varphi:=\{(x,\alpha)\in X\times\mathbb{R}: \varphi(x)=\alpha\},
 $$
 $$
 \text{graph}A:=\{(x,x^*)\in X\times X^*: x^*\in A(x)\}.
 $$ 
 A subset $U$ of $X$ is convex if it contains all closed segments connecting two points in $U$.
The function $\varphi$ is said to be convex if the domain of $\varphi$ is convex and for any $\alpha\in [0,1]$, $x,y\in $ dom$\varphi$ we always have the inequality $\varphi(\alpha x+ (1-\alpha)y)\le \alpha\varphi(x)+(1-\alpha)\varphi(y)$. 

\medskip
As usual, the Fr\'echet subdifferential of  a proper lower semicontinuous function $\varphi$ is the set-valued mapping $\widehat{\partial}\varphi: X \rightrightarrows X^*$ defined by
$$
\widehat{\partial}\varphi(x):=\left\{x^*\in X^*: \liminf_{y\to x}\frac{\varphi(y)-\varphi(x)-\langle x^*,y-x\rangle}{\|y-x\|}\ge 0\right\},\; \text{ for all } x\in \text{ dom}\varphi.
$$
When $\varphi$ is convex, the Fr\'echet subdifferential reduces to the convex analysis subdifferential 
$$
\widehat{\partial}\varphi(x)=\partial \varphi(x):=\{x^*\in X^*: \langle x^*,y-x\rangle\le \varphi(y)-\varphi(x)\},\; \text{ for all } x\in \text{ dom}\varphi.
$$
An operator $A$ is monotone if for all $x,y \in$ dom$A$, one has $\langle x^*-y^*,x-y\rangle\ge 0$ with $x^*\in A(x), y^*\in A(y)$.
It is well-known that when $\varphi$ is convex, the operator $\widehat\partial\varphi$ is monotone \cite{Rock70}. The inverse implication also holds in Asplund space \cite[Theorem 3.56]{Mordukhovich06}; but it is not true in general Banach spaces. The reader is referred to the proof of the reverse implication in \cite[Theorem 2.4]{Trang12} for a counter-example.

\medskip
Let us recall some notions of generalized convex functions. 
\begin{Definition}\label{Def}{\rm
		A function $\varphi: X\to \overline{\mathbb{R}}$ is
		\begin{enumerate}
			\item \textbf{quasiconvex} if 
			\begin{equation}
			\label{quasi}
			\forall x,y\in X, \lambda\in ]0,1[,\quad f(\lambda x+(1-\lambda)y)\le \max\{f(x),f(y)\}.
			\end{equation}
			\item \textbf{$\alpha$-robustly quasiconvex} with $\alpha>0$ if, for every $v^*\in\alpha\mathbb{B}^*$, the function $\varphi_{v^*}:x\mapsto \varphi(x)+\langle v^*,x\rangle$ is quasiconvex. 
		\end{enumerate}
	}
\end{Definition}
Clearly, $\varphi$ is $\alpha$-robustly quasiconvex iff  the function $\varphi_{v^*}$ is quasiconvex for all $v^*\in X^*$ such that $\|v^*\|<\alpha$.

Tracing back to the original definition of robustly quasiconvex functions, they were first defined in \cite{PhuAn96} under the name ``s-quasiconvex'' or ``stable quasiconvex'', and then renamed ``robustly quasiconvex'' in \cite{BarronGoebelJensen12}. This class of functions holds a notable role, as many important optimization properties of generalized convex functions are stable when disturbed by a linear functional with a sufficiently small norm (for instance, all lower level sets are convex, each minimum is global minimum, each stationary point is a global minimizer). For interested readers, we refer to \cite{PhuAn96} again, and further related works  \cite{An061,BarronGoebelJensen12}.
\begin{Definition}
	An operator $A: X\rightrightarrows X^*$ is quasimonotone if for all $x,y\in X$ and $x^*\in A(x),y^*\in A(y)$ we have $\min\{\langle x^*,y-x\rangle,\langle y^*,x-y\rangle\}\leq 0$.
\end{Definition}
Significant contributions concerning dual criteria for quasiconvex functions are in \cite{AusCorMar94,Aus98}. Those characterizations are applicable for a wide range of subdifferentials, for instance Rockafellar-Clarke subdifferentials in Banach spaces, and Fr\'echet subdifferentials in reflexive spaces. These results are still unclear for Fr\'echet subdifferentials in Asplund spaces. Below, we give a short proof to clarify this. Our proof relies on the proof scheme of \cite{AusCorMar94} and the following approximate mean value theorem \cite[Theorem 3.49]{Mordukhovich06}.
\begin{Theorem}\label{M.v.t}
	Let $X$ be an Asplund space and $\varphi: X\to \overline{\mathbb{R}}$ be a proper lower semicontinuous function finite at two given points $a\neq b$. Consider any point $c\in [a,b)$ at which the function 
	$$
	\psi(x):=\varphi(x)-\frac{\varphi(b)-\varphi(a)}{\|a-b\|}\|x-a\|
	$$
	attains its minimum on $[a,b]$; such a point always exists. Then, there are sequences $x_k\overset{\varphi}{\to} c$ and $x^*_k\in \widehat{\partial}\varphi(x_k)$ satisfying
	\begin{equation}\label{mean_1}
	\liminf_{k\to \infty}\langle x^*_k,b-x_k\rangle\ge \frac{\varphi(b)-\varphi(a)}{\|a-b\|}\|b-c\|,
	\end{equation}
	\begin{equation}\label{mean_2}
	\liminf_{k\to \infty} \langle x_k^*,b-a\rangle\ge \varphi(b)-\varphi(a).	
	\end{equation}
	Moreover, when $c\neq a$ one has 
	\begin{equation}\label{mean_3}
	\lim_{k\to \infty}\langle x^*_k,b-a\rangle=\varphi(b)-\varphi(a).
	\end{equation}
\end{Theorem}
Theorem~\ref{M.v.t} allows us to deduce the following three-points lemma which is similar to \cite[Lemma 3.1]{AusCorMar95}. 
\begin{Lemma} \label{Distance}
Let $\varphi:X\rightarrow\overline{\mathbb{R}}$ be a proper, lower semicontinuous function on an Asplund space $X$. Let $u,v,w\in X$ such that $v\in[u,w]$, $\varphi(v)>\varphi(u)$ and $\lambda>0$. Then, there are $\bar{x}\in \text{dom}\varphi$ and $\bar{x}^*\in\widehat{\partial}\varphi(\bar{x})$ such that
$$
\bar{x}\in \mathbb{B}_{\lambda}([u,v]) \quad\text{and}\quad \langle \bar{x}^*,w-\bar{x}\rangle>0,
$$
where
$$
\mathbb{B}_{\lambda}([u,v]):=\{x\in X: \exists y\in [u,v] \;\text{such that}\;\|x-y\|<\lambda\}.
$$
\end{Lemma}

We are in position to establish characterizations of quasiconvexity in terms of Fr\'echet subdifferentials in Asplund spaces.

\begin{Theorem}\label{fod} Let $\varphi:X\rightarrow\overline{\mathbb{R}}$ be a proper lower semicontinuous function on an Asplund space $X$. The following statements are equivalent
	\begin{itemize}
		\item[{\rm(a)}] $\varphi$ is quasiconvex;
		\item[{\rm(b)}] If there are $x,y\in X$ such that $\varphi(y)\le\varphi(x)$, then $\langle x^*, y-x\rangle \le 0$ for all $x^*\in \widehat\partial \varphi(x)$.
		\item[{\rm(c)}]  $\widehat{\partial}\varphi$ is quasimonotone.
	\end{itemize}
\end{Theorem}
\begin{proof}
	(a)$\Rightarrow$(b) Assume that $x,y\in X$, $\varphi(x)\ge \varphi(y)$, and $x^*\in \widehat\partial \varphi(x)$. Consider $S_x:=\{u\in X: \varphi(u)\le\varphi(x)\}$. Since $\varphi$ is quasiconvex, then $S_x$ is a convex set. Thus, we have the function $f:=\delta_{S_x}+\varphi(x)$ is convex, where $\delta_{S_x}$ is equal to $0$ for $u\in S_x$ and to $\infty$ otherwise.
	On the other hand, $f(x)=\varphi(x)$ and $f(u)\ge \varphi(u)$ for all $u\in X$, thus $\widehat\partial \varphi(x)\subset \widehat\partial f(x)$. By the definition of convex subdifferential, since $x^*\in \widehat\partial \varphi(x)\subset \widehat\partial f(x)$, we have $\langle x^*, y-x\rangle \le 0$.
	
	\medskip\noindent
	(b)$\Rightarrow$(c) Assume that there are $x,y\in X$ and $x^*\in \widehat\partial \varphi(x)$, $y^*\in \widehat\partial \varphi(y)$ such that $\langle x^*,x-y\rangle<0$ and $\langle y^*,x-y\rangle>0$. Then, by (b),  $\varphi(x)<\varphi(y)$ and $\varphi(y)<\varphi(x)$, which is a contradiction.
	
	\medskip\noindent
	(c)$\Rightarrow$(a) By using Lemma~\ref{Distance}, the proof of this assertion is similar to one in \cite[Theorem 4.1]{AusCorMar94}.
\end{proof}
\begin{Remark}
{\rm Observe that the implications $(a)\Rightarrow(b)$ and $(b)\Rightarrow(c)$ hold in Banach spaces while $(c)\Rightarrow(a)$ only holds in Asplund spaces. In fact, the equivalence of these statements actually can characterize the Asplund property in the sense that if $X$ is not an Asplund space, then there is a function $\varphi$ whose Fr\'echet subdiferential satisfies (b) and (c) but is not quasiconvex. Such a function $\varphi$ can be found in \cite[Theorem~2.4]{Trang12}.
}
\end{Remark}
\section{Characterizations of Robustly Quasiconvex Functions}
A zero and first order characterization of robust convexity was given in  \cite[Proposition 5.3]{BarronGoebelJensen121} for finite dimensional spaces. We remark that there is an oversight in the proof given there; although the function $f$ is only assumed to be lower semicontinuous, the existence of $z$ in the second paragraph actually requires continuity. Here we show that this conclusion is still correct not only when $f$ is assumed just to be lower semicontinuous, but also when $X$ is only assumed to be an Asplund space. To derive this generalization, we need the following lemmas, revealing that  quasiconvex functions have certain nice properties which resemble those of convex functions.
\begin{Lemma}\label{ltttheotia}
If $\varphi: X\rightarrow\overline{\mathbb{R}}$ is a quasiconvex and lower  semicontinuous function, and $u, v\in X$ are such that $\varphi(v)\geq \varphi(u)$ then
	\begin{equation}\label{QuasiLimit}
	\lim_{t \downarrow 0} \varphi(v+ t(u - v)) = \varphi(v).
	\end{equation}
\end{Lemma}
\begin{proof}
	Suppose that $u,v\in X$ and that $\varphi(v)\geq\varphi(u)$. Since $\varphi$ is quasiconvex, for all $t\in ]0,1[$, we have
	$
	\varphi(v+t(u-v)) \le \max\{\varphi(v),\varphi(u)\}=\varphi(v).
	$		
	It follows that $\limsup_{t \downarrow 0} \varphi(v+t(u-v)) \le \varphi(v).$
Combining the latter with the lower semicontinuity of $\varphi$ we get \eqref{QuasiLimit}.
	$\hfill\Box$
\end{proof}
\begin{Lemma}\label{Bode2}
	Let $\varphi: X\rightarrow\overline{\mathbb{R}}$ be a quasiconvex function and $u,v,w \in X$ such that $v\in ]u,w[, \varphi(u)\leq\varphi(w) $. Suppose that there exist $v^* \in X^* $ and $z\in ]u,v[$ such that
$\varphi_{v^*}(z) > \max\{\varphi_{v^*}(u),\varphi_{v^*}(w)\}$. Then
	\begin{equation}\label{inequality}
	\varphi(u)<\varphi(z)\leq\varphi(v)\leq\varphi(w).
	\end{equation}
\end{Lemma}
\begin{proof}
	Since $z\in ]u,v[\subset ]u,w[, \varphi(u)\leq\varphi(w)$ and $\varphi $ is quasiconvex we have 
	$\varphi (z) \le \max\{\varphi(u), \varphi(w) \}=\varphi (w)$.
	Hence, the latter and the inequality
	$ \varphi_{v^*}(z)>\varphi_{v^*}(w)$ implies that
	$\langle v^*,z\rangle>\langle v^*,w\rangle$. Again, $z\in ]u,w[$ implies 
 $\langle v^*,z\rangle <\langle v^*,u\rangle$.
	Therefore, the inequality $\varphi_{v^*}(u) < \varphi_{v^*}(z)$ yields 
	$\varphi(u) <  \varphi(z)$. Since $z\in]u,v[$ and $v\in ]z,w[$, we deduce $\varphi(z)\leq\varphi(v)\leq\varphi(w)$ from the latter inequality and the quasiconvexity of $\varphi$. Hence, \eqref{inequality} holds.
	$\hfill\Box$
\end{proof}

\begin{Lemma}\label{Tuacuctri}
	Let $\varphi: X \to \overline{\mathbb{R}} $ be a quasiconvex, proper, and lower semicontinuous function, and $v^*\in X^*$. If $\varphi_{v^*}$ is not quasiconvex then there exist $u,v,w\in X$ such that $v\in ]u,w[$ and 
	\begin{equation}\label{dk3}
	\varphi(w)\geq\varphi(v)>\varphi(u),
	\end{equation}
	\begin{equation}\label{dk1}
	\varphi_{v^*}(v) > \max\{\varphi_{v^*}(u),\varphi_{v^*}(w)\} ,
	\end{equation}
	\begin{equation}\label{dk2}
	\forall \gamma >0, \exists v_\gamma \in \mathbb{B}_{\gamma}(v)\cap ]v,w[\;: \varphi_{v^*}(v)> \varphi_{v^*}(v_\gamma).
	\end{equation}
\end{Lemma}
\begin{proof}
	Since $\varphi_{v^*}$ is not quasiconvex, there exist $u,w\in X$ such that $u\ne w,\varphi(u)\leq\varphi(w)$ and $v_0\in ]u,w[$ such that $\varphi_{v^*}(v_0)>\max \{ \varphi_{v^*}(u), \varphi_{v^*}(w)\}$. Applying Lemma~\ref{ltttheotia}, we get
$\lim_{t \downarrow 0} \varphi(w+ t(u - w)) = \varphi(w)$, and so
$\lim_{t \downarrow 0} \varphi_{v^*}(w+ t(u - w)) = \varphi_{v^*}(w).$
	Since $\varphi_{v^*}(w)<\varphi_{v^*}(v_0)$, there exists $t_0\in ]0,1[$ such that 
	\begin{equation}\label{continuity}
	\varphi_{v^*}(w+ t(u - w))<\varphi_{v^*}(v_0), \quad \forall t\in]0,t_0[.
	\end{equation}
	Consider the set
	$$
	\mathscr{L}:=\{z\in ]u,w[:\varphi_{v^*}(z)\geq \varphi_{v^*}(v_0)\}.
	$$
	Clearly, $\mathscr{L}\ne\emptyset$ and for each $z\in\mathscr{L}$ we have 
	$\|z-w\|\geq t_0\|u-w\|$ by \eqref{continuity}. It follows that
	$$
	r:=\inf\{\|z-w\|:z\in\mathscr{L}\}\in [t_0\|u-w\|,\|u-w\|[\;\subset\;]0,\|u-w\|[,
	$$
$$
v:=w+r\frac{u-w}{\|u-w\|}\in\;]u,w[.
$$
We will show that $v\in \mathscr{L}$ and so  \eqref{dk1} holds. Suppose on the contrary that $v\notin\mathscr{L}$. Then $v_0\in]u,v[$ and  we get $\varphi(u)<\varphi(v_0)\leq\varphi(v)\leq\varphi(w)$ by Lemma~\ref{Bode2}.
	Applying Lemma~\ref{ltttheotia},  we get
$\lim_{t\downarrow 0}\varphi(v+t(u-v))=\varphi(v)$, and so $\lim_{t\downarrow 0}\varphi_{v^*}(v+t(u-v))=\varphi_{v^*}(v)$. By the definition of $r$, there exists 
a sequence $(z_n)\subset\mathscr{L}$ such that $\|z_n-w\|\rightarrow r$ and $\|z_n-w\|>r$ for all $n\in\mathbb{N}$. Therefore,
\begin{eqnarray*}
	\varphi_{v^*}(v)&=&\lim_{t\downarrow 0}\varphi_{v^*}(v+t(u-v))\\
	&=&\lim\varphi_{v^*}\left(v+\frac{\|z_n-w\|-r}{\|u-v\|}(u-v)\right)\\
		&=&\lim\varphi_{v^*}\left(v-\frac{r}{\|u-v\|}(u-v)+\frac{\|z_n-w\|}{\|u-v\|}(u-v)\right)\\
		&=&\lim\varphi_{v^*}\left(v-\frac{r}{\|u-w\|}(u-w)+\frac{\|z_n-w\|}{\|u-w\|}(u-w)\right)\\
		&=&\lim\varphi_{v^*}\left(w+\frac{\|z_n-w\|}{\|u-w\|}(u-w)\right)\\
				&=&\lim\varphi_{v^*}\left(z_n\right)\\
				&\geq&\varphi_{v^*}(v_0),
\end{eqnarray*}
which is a contradiction. Now we show that $v$ satisfies \eqref{dk2}.
Let $\gamma$ be any positive real number and
$$
v_\gamma:=w+\frac{r-r_\gamma}{\|u-w\|}(u-w)\;\text{with}\; r_\gamma:=\min\{r/2, \gamma/2\}>0.
$$
Since $0<r-r_\gamma<r<\|u-w|$, it implies that $v_\gamma\in]v,w[\;\setminus\;\mathscr{L}$. Therefore, $\varphi_{v^*}(v_\gamma)<\varphi_{v^*}(v_0)\leq\varphi_{v^*}(v)$.
Furthermore,
\begin{eqnarray*}
	\|v_\gamma-v\|&=& \left\|w+\frac{r-r_\gamma}{\|u-w\|}(u-w)-w-r\frac{u-w}{\|u-w\|}\right\|=r_\gamma<\gamma.
\end{eqnarray*}
Hence, $v$ satisfies \eqref{dk2}.
$\hfill\Box$	
\end{proof}

\begin{Theorem}\label{Cr} Let $\varphi:X\rightarrow\overline{\mathbb{R}}$ be a proper lower semicontinuous function on a Banach space $X$, and $\alpha >0$. Consider the following statements
	\begin{itemize}
		\item[{\rm(a)}] $\varphi$ is $\alpha-$robustly quasiconvex;
		\item[{\rm(b)}] For every $x,y\in X$
		\begin{align} \label{Robust_sufficient condition 1}
		\varphi(y)\leq\varphi(x)\;\Longrightarrow\;\langle x^*,y-x\rangle\leq-\min\left\{\alpha\|y-x\|,\varphi(x)-\varphi(y)\right\},\; \forall x^*\in\widehat{\partial}\varphi(x).
		\end{align}
		
	\end{itemize}
	Then {\rm (a)}$\Rightarrow${\rm (b)}. Additionally, if $X$ is an Asplund space, then {\rm (b)}$\Rightarrow${\rm (a)}.
\end{Theorem}
\begin{proof} Suppose that $\varphi$ is $\alpha-$robustly quasiconvex, and $x,y\in X$ satisfy $\varphi(y)\leq\varphi(x)$. Assume that $x^*\in \widehat{\partial}\varphi(x)$. We will prove
	$$
	\langle x^*,y-x\rangle\leq-\min\left\{\alpha\|y-x\|,\varphi(x)-\varphi(y)\right\}.
	$$ 
	If $x=y$, the above inequality is trivial.
	Otherwise, we consider two cases:
	
	\medskip\noindent
	\textbf{Case 1.} $\quad\alpha\|y-x\|\leq\varphi(x)-\varphi(y)$
	
	\noindent
	We then need to prove that
	\begin{equation}\label{T32-00}
	\langle x^*,y-x\rangle\leq-\alpha\|y-x\|.
	\end{equation}
	By the Hahn-Banach theorem, there exists $v^*\in X^*$, $\|v^*\|=1$ such that $\langle v^*,y-x\rangle=\|y-x\|$. Consider the function $f:X\rightarrow\overline{\mathbb{R}}$ given by
	$$
	f(z)=\varphi(z)+\alpha\langle v^*,z-x\rangle\quad \forall z\in X.
	$$
	Then $f(x)=\varphi(x)$, and 
	$$
	f(y)=\varphi(y)+\alpha\langle v^*,y-x\rangle	=\varphi(y)+\alpha\|y-x\|\leq\varphi(x)=f(x),
	$$
	i.e., $\max\{f(x),f(y)\}=f(x)$.
Since $\varphi$ is $\alpha-$robustly quasiconvex, $f$ is quasiconvex. 
	Therefore for each $t\in[0,1]$, we always have
	\begin{eqnarray*}
		\varphi(x)=f(x)=\max\{f(x),f(y)\}&\geq&f(x+t(y-x))\\
		&=&\varphi(x+t(y-x))+t\alpha\langle v^*,y-x\rangle\\
		&=& \varphi(x+t(y-x))+t\alpha\|y-x\|,
	\end{eqnarray*}
which implies that
	\begin{equation}\label{T32-0}
	\varphi(x)-t\alpha\|y-x\|\ge \varphi(x+t(y-x)).
	\end{equation}
	Since $x^*\in\widehat{\partial}\varphi(x)$, for any $\gamma>0$, there exists a number $r>0$ such that
	\begin{equation}\label{T32}
	\varphi(z)\geq\varphi(x)+\langle x^*,z-x\rangle-\gamma\|z-x\|\quad \forall z\in \mathbb{B}_r(x).
	\end{equation}
	Let $t\in ]0,1[$ such that  $x+t(y-x)\in \mathbb{B}_r(x)$. It follows from \eqref{T32-0} and \eqref{T32} that
	$$
	\varphi(x)-t\alpha\|y-x\|\geq \varphi(x)+ t\langle x^*,y-x\rangle-t\gamma\|y-x\|,
	$$
and so	
\begin{equation}\label{T32-1}
	\langle x^*,y-x\rangle\leq-\alpha\|y-x\|+\gamma\|y-x\|.
	\end{equation}
	On taking limit  on both sides of the above inequality as $\gamma\rightarrow 0^+$, we get
	\eqref{T32-00}.
	
	\medskip\noindent
	\textbf{Case 2.} $\quad\alpha\|y-x\|>\varphi(x)-\varphi(y)$
	
	\noindent	
	We have $\bar{\alpha}\|y-x\|=\varphi(x)-\varphi(y)$, where 
	$$
	\bar{\alpha}:=\frac{\varphi(x)-\varphi(y)}{\|y-x\|}\in ]0,\alpha[.
	$$
	Since $\varphi$ is 
	$\bar\alpha-$robustly quasiconvex, we derive from Case 1 that
	\begin{eqnarray*}
		\langle x^*,y-x\rangle&\leq&-\bar\alpha\|y-x\|
		=\varphi(y)-\varphi(x)\\
		&=&
		-\min\left\{\alpha\|y-x\|,\varphi(x)-\varphi(y)\right\}.
	\end{eqnarray*}
	
\medskip	
Conversely, assume that $X$ is Asplund, and (b) holds. 
 It follows from Theorem~\ref{fod} that  $\varphi$ is quasiconvex. 
	Suppose that $\varphi$ is not $\alpha-$robustly quasiconvex, i.e., there exists $v^*\in X^*\setminus\{0\}, \|v^*\|<\alpha$ such that $\varphi_{v^*}$ is not quasiconvex. By Lemma~\ref{Tuacuctri}, there are $u,w\in X$ and $v\in]u,w[$ satisfying \eqref{dk3},\eqref{dk1}, and \eqref{dk2}. Since $\varphi_{v^*}(v)>\varphi_{v^*}(u)$, there exists $\delta>0$ such that $\bar v^*:=(1+\delta)v^*$ satisfies $\|\bar v^*\|<\alpha$ and
	$\varphi_{\bar v^*}(v)>\varphi_{\bar v^*}(u)$.
	Thus, we have $\varphi(v)>\varphi(u)$, $\varphi_{v^*}(v)>\varphi_{v^*}(u)$, $\varphi_{\bar v^*}(v)>\varphi_{\bar v^*}(u)$ and the lower semicontinuity of $\varphi, \varphi_{v^*}$, and $\varphi_{\bar{v}^*}$. This implies the existence of $\gamma>0$ satisfying
\begin{equation}\label{inequal}
\varphi(z)>\varphi(u),\quad \varphi_{v^*}(z)>\varphi_{v^*}(u),\quad
\varphi_{\bar v^*}(z)>\varphi_{\bar v^*}(u)\quad \forall z\in \mathbb{B}_\gamma(v).
\end{equation}
	By the assertion \eqref{dk2}, there is $v_\gamma\in\mathbb{B}_\gamma(v)\cap ]v,w[$ such that $\varphi_{v^*}(v)>\varphi_{v^*}(v_\gamma)$. Then, $v_\gamma$ can be written as 
	$$
	v_\gamma:=v+\lambda(w-v)\;
	\text{with}\; \lambda\in \left]0, \min\left\{1,\dfrac{\gamma}{\Vert w-v\Vert}\right\}\right].
	$$
	Since $\varphi_{v^*}(v)>\varphi_{v^*}(w)$ and $\varphi(v)\leq\varphi(w)$, we have
	$\langle v^*,w-v\rangle<0$ and so
	\begin{eqnarray*}
		\varphi_{\bar v^*}(v_\gamma)-\varphi_{\bar v^*}(v)&=&\varphi_{ v^*}(v_\gamma)-\varphi_{v^*}(v)+\delta\langle v^*,v_\gamma-v\rangle\\
		&=&\varphi_{ v^*}(v_\gamma)-\varphi_{v^*}(v)+\delta\lambda\langle v^*,w-v\rangle
		< 0.
	\end{eqnarray*}
	Applying Lemma~\ref{Distance} for $\varphi_{\bar v^*}$, $v\in[v_\gamma, u]$ with $\varphi_{\bar v^*}(v)>\varphi_{\bar v^*}(v_\gamma)$, there exist $x\in\text{dom}\varphi_{\bar v^*}$ and $x^*\in\widehat\partial\varphi_{\bar v^*}(x)$ such that
	\begin{equation}\label{r4}
	x\in[v_\gamma, v]+(r-\|v_\gamma-v\|) \mathbb{B}\quad\text{and}\quad\langle x^*,u-x\rangle>0.
	\end{equation}
	Then $x\in \mathbb{B}_\gamma(v)$ and so $\varphi(x)>\varphi(u)$ by \eqref{inequal}.
	By the assumption (b) and the second inequality of \eqref{r4},
	$$
	-\langle \bar v^*,u-x\rangle<\langle x^*-\bar v^*,u-x\rangle\leq-\min\{\alpha\|u-x\|,\varphi(x)-\varphi(u)\}.
	$$
	Since $\langle \bar v^*,u-x\rangle\leq\|\bar v^*\|\|u-x\|<\alpha\|u-x\|$, the above inequality implies that
	$\langle \bar v^*,u-x\rangle>\varphi(x)-\varphi(u)$, i.e.,  $\varphi_{\bar v^*}(x)<\varphi_{\bar v^*}(u)$ and this contradicts \eqref{inequal}.
$\hfill\Box$
\end{proof}

\medskip
We next construct a completely new characterization for the robust quasiconvexity. It is based on the equivalence of the quasiconvexity of a lower semicontinuous function and the quasimonotonicity of its subdifferential operator.

\begin{Theorem}\label{Dactrunglqdondieu}
	Let $\varphi:X\rightarrow\overline{\mathbb{R}}$ be proper, lower semicontinuous on an Asplund space $X$ and $\alpha>0$. Then, $\varphi$ is $\alpha-$robustly quasiconvex if and only if for any 
	$(x,x^*),(y,y^*)\in \text{graph }{\widehat{\partial}}\varphi$, we have
	\begin{equation}\label{Characterizarion}	
	\min\{\langle x^*,y-x\rangle, \langle y^*,x-y\rangle\}> -\alpha\Vert y-x\Vert \Longrightarrow \langle x^*-y^*, x-y\rangle \geq 0.
	\end{equation}
\end{Theorem}
\begin{proof} Suppose that $\varphi$ is $\alpha-$robustly quasiconvex and that there exist $(x,x^*),(y,y^*)\in $ graph $\widehat{\partial}\varphi$ such that
	\begin{equation}
	\label{e}
	\min\{\langle x^*,y-x\rangle, \langle y^*,x-y\rangle\}> -\alpha\Vert y-x\Vert.
	\end{equation}
	Since $\varphi$ is quasiconvex, $\widehat\partial\varphi$ is quasimonotone by Theorem~\ref{fod}. It follows that
	\begin{equation}\label{Characterization_Quasiconvex}
	\min\{\langle x^*,y-x\rangle, \langle y^*,x-y\rangle\}\leq 0.
	\end{equation}	
	Combining \eqref{e} and \eqref{Characterization_Quasiconvex}, we have
	$$
	0\leq-\min\left\{\left\langle x^*,\frac{y-x}{\|y-x\|}\right\rangle, \left\langle y^*,\frac{x-y}{\|x-y\|}\right\rangle\right\}<\alpha.
	$$
	Without loss of generality, we may assume 
	$$
	\left\langle x^*,\frac{y-x}{\|y-x\|}\right\rangle=\min\left\{\left\langle x^*,\frac{y-x}{\|y-x\|}\right\rangle, \left\langle y^*,\frac{x-y}{\|x-y\|}\right\rangle\right\}.
	$$
	Let $r>0$ be such that
	\begin{equation}\label{Chossing_r}
	-\left\langle x^*,\frac{y-x}{\|y-x\|}\right\rangle<r\leq\alpha.
	\end{equation}
	By the Hahn-Banach theorem, there exists $v^*\in X^*$ satisfying
$\langle v^*, y-x\rangle = r\Vert y-x\Vert$ and $\Vert v^*\Vert=r\leq\alpha$.
It follows that
\begin{equation}\label{22}
\langle x^*,y-x\rangle +\langle v^*,y-x\rangle > -r\|y-x\|+r\|y-x\|=0.
\end{equation}
	Consider $\varphi_{v^*}:X\rightarrow\overline{\mathbb{R}}$ given by
	$\varphi_{v^*}(u)=\varphi(u)+\langle v^*,u\rangle$ for any $u\in X$.
	Then, we have
	$
	\widehat{\partial}\varphi_{v^*}(u)=\widehat{\partial}\varphi(u)+v^*$ for $u\in\text{dom}\varphi$. Hence, by the quasiconvexity of $\varphi_{v^*}$ and by Theorem~\ref{fod}, we have
	$$
	\min \{\langle x^*, y-x\rangle+\langle v^*, y-x\rangle, \langle y^*, x-y\rangle+\langle v^*, x-y\rangle\}\leq 0.
	$$
	Combining with \eqref{22}, it implies
	$$
	\langle y^*, x-y\rangle+\langle v^*, x-y\rangle\leq 0,\;\text{i.e.},\;
	\langle y^*, x-y\rangle\leq \langle v^*,y-x\rangle=r\Vert x-y\Vert.
	$$	
	Letting
	$
	r\rightarrow -\left\langle x^*,\frac{y-x}{\|y-x\|}\right\rangle
	$,
	we obtain $\langle y^*, x-y\rangle\leq \langle x^*,x-y\rangle$ and thus \eqref{Characterizarion} holds.
	
	\medskip
	Conversely, assume that \eqref{Characterizarion} holds for all $x,y\in X$ and
	$x^*\in\widehat{\partial}\varphi(x), y^*\in\widehat{\partial}\varphi(y)$.
	Taking any $v^*$ in $\alpha \mathbb{B}^*$, we next prove that  $\varphi_{v^*}:X\rightarrow\mathbb{R}$, defined by 
	$\varphi_{v^*}(u)=\varphi(u)+\langle v^*,u\rangle$ for any $u\in X$,
	is quasiconvex by showing the quasimonotonicity of $\widehat\partial\varphi_{v^*}$.
	Taking any $x,y\in X$ and $x^*\in \widehat{\partial}\varphi_{v^*}(x)$, $y^*\in\widehat{\partial}\varphi_{v^*}(y)$, then
 $x^*-v^*\in \widehat{\partial}\varphi(x)$, $y^*-v^*\in \widehat{\partial}\varphi(y)$. 
	We then consider two cases.

\medskip\noindent
\textbf{Case 1}. $\quad\min\{\langle x^*-v^*,y-x\rangle, \langle y^*-v^*,x-y\rangle\}\leq -\alpha\Vert y-x\Vert$
		
		Without loss of generality, assume that
		$$
		\langle x^*-v^*,y-x\rangle =\min\{\langle x^*-v^*,y-x\rangle, \langle y^*-v^*,x-y\rangle\}.
		$$
		Since $\Vert v^*\Vert\leq \alpha$, we have
		\begin{eqnarray*}
			\min\{ \langle x^*,y-x\rangle, \langle y^*, x-y\rangle\}&\leq&\langle x^*,y-x\rangle
			=\langle x^*-v^*,y-x\rangle +\langle v^*, y-x\rangle\\
			&\leq& -\alpha \Vert y-x\Vert +\Vert v^*\Vert \Vert y-x\Vert
			\leq 0.
		\end{eqnarray*}
\textbf{Case 2}. $\quad\min\{\langle x^*-v^*,y-x\rangle, \langle y^*-v^*,x-y\rangle\}> -\alpha\Vert y-x\Vert$
		
		Since \eqref{Characterizarion} is satisfied, we have 
		$$
		\langle (x^*-v^*)-(y^*-v^*), x-y\rangle \geq 0,
		$$
		i.e., $\langle x^*-y^*, x-y\rangle \geq 0$. It implies that
		\begin{eqnarray*}
			2\min\{ \langle x^*,y-x\rangle, \langle y^*, x-y\rangle\}\leq\langle x^*,y-x\rangle+\langle y^*, x-y\rangle\leq 0.
		\end{eqnarray*}
Hence, $\widehat{\partial}\varphi_{v^*}$ is quasimonotone and thus $\varphi_{v^*}$ is quasiconvex for any $v^*\in\alpha \mathbb{B}^*$ by Theorem~\ref{fod}. This yields the $\alpha$-robust quasiconvexity of $\varphi$.
	$\hfill\Box$
\end{proof}

\section{Conclusions}
Using Fr\'echet subdifferentials, we have obtained two first-order characterizations for the robust quasiconvexity of lower semicontinuous functions in Asplund spaces.  The first one is a generalization of \cite[Proposition~5.3]{BarronGoebelJensen121} from finite dimensional spaces to Asplund spaces and its proof also overcomes a glitch in the proof of the sufficient condition of \cite[Proposition~5.3]{BarronGoebelJensen121}. The second criterion is totally new and it is settled from the equivalence of the quasiconvexity of lower semicontinuous functions and the quasimonotonicity
of their subdifferential operators. Further investigations are needed to apply those characterizations in partial differential equations with connections to differential geometry, mean curvature, tug-of-war games, and stochastic optimal control \cite{BarronGoebelJensen12,BarronGoebelJensen121,BarronGoebelJensen13}.

\subsection*{Acknowlegement}
This work was completed while the second author was
visiting Vietnam Institute for Advanced Study in Mathematics (VIASM). He would like to thank VIASM for the very kind support and hospitality.

\end{document}